\documentclass[12pt]{amsart}

\usepackage{breqn}
\usepackage[english]{babel}
\usepackage{graphicx}
\usepackage{framed}
\usepackage[normalem]{ulem}
\usepackage{amsmath}
\usepackage{amsthm}
\usepackage{amssymb}
\usepackage{amsfonts}
\usepackage{enumerate}
\usepackage[utf8]{inputenc}
\usepackage[top=1 in,bottom=1in, left=1 in, right=1 in]{geometry}
\usepackage{comment}
\RequirePackage{xr-hyper}
\RequirePackage[colorlinks=true]{hyperref}

\newcommand{\trace}{{\rm trace}}

\newtheorem{theorem}{Theorem}
\newtheorem{corollary}{Corollary}
\newtheorem*{definition}{Definition}

\newtheorem{prop}{Proposition}
\newtheorem{lemma}{Lemma}

\newtheorem{rmk}{Remark}
\theoremstyle{definition}

\makeatletter
\def\pmod#1{\allowbreak\mkern2mu({\operator@font mod}\,\,#1)}
\makeatother

\makeatletter
\@namedef{subjclassname@2020}{%
\textup{2020} Mathematics Subject Classification}
\makeatother

\begin{document}

\title{On sign changes of primitive Fourier coefficients of Siegel cusp forms}

\author{Karam Deo Shankhadhar, Prashant Tiwari}

\address[Karam Deo Shankhadhar]{Department of Mathematics, Indian Institute of Science Education and Research Bhopal,
Bhopal Bypass Road,  Bhauri,
Bhopal 462 066,
Madhya Pradesh, India}
\email{karamdeo@iiserb.ac.in}

\address[Prashant Tiwari]{Department of Mathematics, Indian Institute of Science Education and Research Bhopal,
Bhopal Bypass Road,  Bhauri,
Bhopal 462 066,
Madhya Pradesh, India}
\email{ptiwari@iiserb.ac.in}




\begin{abstract}
In this article, we establish quantitative results for sign changes in certain subsequences of primitive Fourier coefficients of a non-zero Siegel cusp form of arbitrary degree over congruence subgroups. As a corollary of our result for degree two Siegel cusp forms, we get sign changes of its diagonal Fourier coefficients. In the course of our proofs, we prove the non-vanishing of certain type of Fourier-Jacobi coefficients of a Siegel cusp form and all theta components of certain Jacobi cusp forms of arbitrary degree over congruence subgroups, which are also of independent interest.      
\end{abstract}

\maketitle

\section{Introduction}

Let $k, N$ be positive integers and $\chi$ be a Dirichlet character modulo $N$. Let $g \geq 2$ and 
$S_k^g(N, \chi)$ denote the space of Siegel cusp forms of weight $k$, degree $g$ and Dirichlet character 
$\chi$ over the congruence subgroup $\Gamma_0^g(N) (\subseteq Sp_g(\mathbb{Z}))$ (for the precise definition and other properties of Siegel cusp forms, we refer to \S 2 below). Any $F\in S_k^g(N,\chi)$ has a unique Fourier series expansion 
\begin{equation*}
F(Z)=\sum_{T\in\mathcal{J}_g} a_F(T) e^{2 \pi i \ {\rm trace} (TZ)}, 
\end{equation*}
where $Z$ is in the Siegel upper-half space of degree $g$ and $\mathcal{J}_g$ denotes the set of half-integral, symmetric, positive-definite $g \times g$ matrices. The Fourier coefficients $a_F(T)$ of $F$ are quite mysterious 
objects and have been investigated by several authors over the past years from various aspects. 

An important direction would be to look for small and important subsets of $\mathcal{J}_g$ such that any Siegel cusp form of degree $g$ must have non-zero Fourier coefficients supported on it. In \cite[p. 387]{DZ}, Zagier 
showed that any non-zero Siegel cusp form of degree $2$ over the group $Sp_2(\mathbb{Z})$ is uniquely determined by the Fourier coefficients supported on primitive matrices (for a precise definition of primitive matrices, see \S 2). Yamana \cite{SY} generalized this result for any degree $g \geq 2$ and any congruence subgroup $\Gamma_0^g(N)$. In \cite{saha}, Saha considered a proper subset of primitive matrices containing those matrices $T \in \mathcal{J}_2$ such that $4 \det T$ is odd, square-free and proves that any non-zero
Siegel cusp form of degree $2$ over $Sp_2(\mathbb{Z})$ has non-zero Fourier coefficients supported on it.
B\"ocherer and Das \cite{BD} recently established a  quantitative result extending Saha's result to vector valued  
Siegel modular forms of arbitrary degree over the group $Sp_g(\mathbb{Z})$. In \cite{SS}, Saha's result was extended to the congruence subgroups $\Gamma_0^2(N), N$ square-free.      
In a recent work \cite{M}, Martin presented a set of diagonal matrices such that any Siegel 
cusp form in the space $S_k^2(N, \chi)$ ($k$ even, $N$ odd, square-free and $\chi$ primitive) has Fourier 
coefficients supported on it. In this paper, we are interested in understanding the sign changes of the 
primitive Fourier coefficients of Siegel cusp forms.    

In \cite{SJ}, it is proved that if the Fourier coefficients $a_F(T)$ of a non-zero Siegel cusp form $F$ of even integral weight over the symplectic group $Sp_{g}(\mathbb{Z}) (g \geq 2)$ are real then there are infinitely 
many $T \in \mathcal{J}_g$ (modulo the usual action of $GL_g(\mathbb{Z})$) such that $a_F(T) > 0$ and
similarly infinitely many $T$ such that $a_F(T) < 0$. 
Moreover, Choie, Gun and Kohnen \cite{CGK} gave an explicit upper bound for the first sign change of the Fourier coefficients of these Siegel cusp forms. The bound obtained in this paper was later improved by 
He and Zhao \cite{HZ} by strengthening a result used in \cite {CGK} from elliptic modular forms of integral weight.
In \cite{GS17}, Gun and Sengupta obtained certain quantitative results for sign changes of Fourier coefficients 
of Siegel cusp forms in the space $S_k^2(N)$ with $k$ even and $N$ square-free. In \cite{CGK}, Choie, Gun and Kohnen ask more generally about the distribution of signs of the primitive Fourier coefficients. In this paper, we investigate sign changes of Fourier coefficients of Siegel cusp forms of arbitrary degree $g (\geq 2)$ and for the congruence subgroup $\Gamma_0^g(N), N$ odd, supported on certain sparse subsets of $\mathcal{J}_g$ which are contained in the set of primitive $GL_g(\mathbb{Z})$-inequivalent matrices. 

\subsection{Statement of the main results}
For any positive integer $N$ and any Dirichlet character $\chi$ modulo $N$ with conductor $m_\chi$, 
we denote by 
$S_k^{g, {\rm old}} (N,\chi)$, the linear subspace of $S_k^g(N, \chi)$ spanned by the set
$$
\{F(dZ) | F\in S_k^g(M, \chi), d \in \mathbb{Z}_{>0}, m_\chi | M, dM | N, M \neq N\}.
$$ 
We denote by $S_k^{g, {\rm new}}(N,\chi)$, the orthogonal complement (with respect to the Petersson inner product) of $S_k^{g, {\rm old}} (N,\chi)$ in $S_k^g(N, \chi)$. Throughout the paper, by a {\textit{newform}} in 
$S_k^g(N, \chi)$ we mean an element of $S_k^{g, {\rm new}}(N,\chi)$. Note that, it is sufficient to consider the sign changes for the newforms as the problem in the complement space $S_k^{g, {\rm old}} (N, \chi)$ reduces 
to newforms of lower levels.

In order to state our results, let us introduce some notations which will be used throughout the article. 
Let $n \geq 1$. For any $n \times n$ symmetric, half-integral, positive definite matrix $\mathcal{M}$ and 
any $\mu \in \mathbb{Z}^{n,1}$, consider the following subset $S_{\mathcal{M}, \mu}$ of $\mathcal{J}_{n+1}$.  
\begin{equation*}\label{eq:Mmu}
S_{\mathcal{M}, \mu} = \left\{T=\begin{pmatrix} \frac{m+\mathcal{M}^*[\mu]}{4 \det \mathcal{M}} & \mu^t/2 \\ \mu/2 & \mathcal{M} \end{pmatrix}:
m \in \mathbb{Z}\ {\rm such\ that}\ T \in \mathcal{J}_{n+1}\right\},
\end{equation*}
where $\mathcal{M}^*$ denotes the cofactor matrix of $\mathcal{M}$, $\mu^t$ denotes the transpose of $\mu$  and $\mathcal{M}^*[\mu]=\mu^t \mathcal{M}^* \mu$. If $n=1$ then we define $\mathcal{M}^*=1$. Note that 
$4 \det T = m$. 

\begin{theorem}\label{thm:1}
Let $g, k, N$ be positive integers. Assume that $g \geq 2$ and $N$ is an odd positive integer. 
Let $\chi$ be a Dirichlet character modulo $N$ such that $\chi(-1)=(-1)^k$. 
Let $F \in S_k^{g+1, {\rm new}}(N,\chi)$ be a non-zero 
Siegel cusp form with 
real Fourier coefficients $a_F(T)$. Then there are infinitely many
half-integral, positive definite, primitive $g \times g$ matrices $\mathcal{M}$, and for any such fixed $\mathcal{M}$ there is  
a $\mu \in {\mathbb{Z}^{g, 1}}$ such that
the sequence of Fourier coefficients $\{a_F(T): T \in S_{\mathcal{M},\mu}\}$
has at least one sign change in the interval $4 \det T \in (x, x+x^{3/5}]$ for $x \gg 1$. 
\end{theorem}

\begin{theorem}\label{thm:3}
Let $k, N, \chi$ be as in Theorem \ref{thm:1}. Let $F \in S_k^{2, {\rm new}}(N,\chi)$ 
be a non-zero Siegel cusp form with real Fourier coefficients $a_F(T)$. 
Then there are infinitely many odd primes $p$ such that 
for any such fixed prime $p$ and for any fixed $\mu \in \mathbb{Z}$, the sequence of Fourier coefficients 
$\{a_F(T):T \in S_{p, \mu}\}$ 
changes sign at least once
for $4 \det T \in (x, x+x^{3/5}]$ for $x \gg 1$.  
\end{theorem}

By taking $\mu=0$ in Theorem \ref{thm:3}, we get the sign changes of the diagonal coefficients. 
More precisely, we have the following corollary.
\begin{corollary}\label{cor:1}
Let $F \in S_k^{2, {\rm new}}(N,\chi)$ be as in Theorem \ref{thm:3}. Then there are infinitely many odd primes 
$p$ such that for any such fixed prime $p$, the Fourier coefficients 
$a_F\left(\begin{pmatrix} n & 0 \\ 0 & p \end{pmatrix}\right)$
change sign at least once for $n \in (x, x+x^{3/5}]$ for $x \gg 1$.
\end{corollary}

Suppose $F$ is a non-zero Siegel-Hecke eigenform of even weight for the group $Sp_2(\mathbb{Z})$.  
By using \cite[Theorem 1]{MM} we get that the first Fourier-Jacobi coefficient of $F$ is non-zero and then by following the method of this paper one can easily see that the odd prime $p$ can be replaced by $1$ in 
Theorem \ref{thm:3} and Corollary \ref{cor:1}. Hence, in this case we get the sign changes of the primitive, diagonal $GL_2(\mathbb{Z})$-inequivalent Fourier coefficients.

\begin{corollary}
Let $F$ be a non-zero Siegel-Hecke eigenform of even weight for the group $Sp_2(\mathbb{Z})$. 
Then the Fourier coefficients 
$a_F\left(\begin{pmatrix} n & 0 \\ 0 & 1 \end{pmatrix}\right)$
change sign at least once for $n \in (x, x+x^{3/5}]$ for $x \gg 1$.
\end{corollary}

In a recent article \cite{ALS}, Asaari, Lester and Saha discuss in detail the sign changes of the Fourier 
coefficients $a_F(T)$ of Siegel cusp forms of degree $2$ and odd, square-free level for $
T \in \mathcal{J}_2$ such that $4 \det T$ is odd and square-free.   

\begin{rmk}
In Theorem \ref{thm:1}, for fixed $\mathcal{M}$ and $\mu$ the considered matrices 
$T \in S_{\mathcal{M}, \mu}$ for sign changes 
are primitive and $GL_g(\mathbb{Z})$-inequivalent. Similarly, in Theorem \ref{thm:3}
the considered matrices in the set $S_{p, \mu}$ are $GL_2(\mathbb{Z})$-inequivalent. 
Moreover, by choosing $\mu$ coprime to $p$, we get sign changes of the primitive 
$GL_2(\mathbb{Z})$-inequivalent Fourier coefficients.
\end{rmk}

\begin{rmk}
If $F$ is a Siegel cusp form with Fourier coefficients $a_F(T)$ then the Fourier series with coefficients 
Re$(a_F(T))$ or Im$(a_F(T))$ are also Siegel cusp forms. Therefore Theorem \ref{thm:1} and
Theorem \ref{thm:3} can be reformulated for 
real and imaginary parts of the Fourier coefficients $a_F(T)$ if they are not assumed to be real.
\end{rmk}

This article is organized as follows. In the next section we set the notations, recall the definition of Siegel cusp forms, discuss its Fourier-Jacobi decomposition, and the theta decomposition of Jacobi cusp forms. 
In \S 3 we state our intermediate results about Jacobi forms which are useful to prove our main theorems and 
are of independent interest as well. 
In \S 4 we establish some lemmas from the theory of quadratic forms which are used to prove the intermediate results. In \S 5 we prove all our results about Jacobi forms, and then we establish 
Theorems \ref{thm:1} and \ref{thm:3} in \S 6.

\section{Notations and Preliminaries}

Let $R$ be a commutative ring with identity $1$. For any two positive integers $m$ and $n$, let $R^{m,n}$ denotes the space of $m \times n$ matrices with entries in $R$. If $m=1$ we simply write $R^n$ in place of 
$R^{1,n}$ and also we use the notation $M_n(R)$ for $R^{n,n}$. Let $I_n$ denote the identity matrix of size 
$n$. We denote the set of all $n \times n$ symmetric, half-integral, positive definite matrices in 
$M_n(\mathbb{Q})$ by $\mathcal{J}_n$. For any $X \in M_n(\mathbb{C})$, let $e(X):=e^{2\pi i \ \trace(X)}$.
For any two matrices $A$ and $X, A[X] := X^tAX (X^t$ refers to the transpose of $X)$. We denote the cofactor matrix of any square matrix $A$ by $A^*$.

We say that a matrix $A \in M_n(\mathbb{Q})$ is {\it half-integral} if it has integral diagonal entries and half-integral off-diagonal ones. Following \cite{SY}, define the content of a half-integral matrix $A$ by
$$\max\{a \in \mathbb{Z} : a^{-1}A \ {\text{is half-integral}}\}.$$ 
We call half-integral matrices of content one {\it primitive}.

Let $n$ be a positive integer and $Sp_n(\mathbb{Z})$ be the Siegel modular group of degree $n$
$$
Sp_n(\mathbb Z)=\left\{ M \in M_{2n}(\mathbb Z) : J [M] = J, J = \begin{pmatrix} 0 & I_n\\ -I_n & 0 \end{pmatrix}\right\},
$$
and $\mathbb{H}_n$ be the Siegel upper-half space of degree $n$
$$\mathbb{H}_n =\{ Z \in M_n(\mathbb C) : Z=Z^t, {\rm Im} (Z) \ {\text {is positive definite}}\},$$
where ${\rm Im} (Z)$ denotes the imaginary part of the matrix $Z$. 
The group $Sp_n(\mathbb{Z})$ acts on $\mathbb{H}_n$ as follows. 
$$M Z= (AZ+B) (CZ+D)^{-1}, M=\begin{pmatrix} A & B\\ C & D \end{pmatrix} \in Sp_n(\mathbb{Z}), 
Z \in \mathbb{H}_n.$$ 
For any positive integer $N$, let 
$$
\Gamma_0^n(N) = \left\{\begin{pmatrix} A&B\\C&D \end{pmatrix} \in Sp_n(\mathbb{Z}) | C \equiv 0 \pmod{N}      \right\}
$$
be the Hecke congruence subgroup of $Sp_n(\mathbb{Z})$ of level $N$. Let $\chi$ be any Dirichlet character 
modulo $N$. We write $\chi(M) = \chi(\det D)$ whenever $M = \begin{pmatrix} A&B\\C&D \end{pmatrix} \in 
\Gamma_0^n(N)$. 

\begin{definition}
Let $n \geq 2$. Let $k, N$ be positive integers and $\chi$ be a Dirichlet character modulo $N$. A {\it Siegel modular form} of weight $k$, level $N$, degree $n$ with character $\chi$ is a holomorphic function $F : \mathbb{H}_n \rightarrow \mathbb{C}$ such that 
$$F|_k M (Z) := \det (CZ+D)^{-k} F(M Z)=\chi(M)F(Z), \ {\rm for \ all} \ M=\begin{pmatrix} A & B \\ C & D \end{pmatrix} \in \Gamma_0^n(N).$$ 
Moreover, we say that $F$ is a cusp form if it vanishes at every cusp of $\Gamma_0^n(N)$. 
\end{definition}

We denote by $M_k^n (N, \chi)$, the finite dimensional $\mathbb{C}$-vector space of all Siegel modular forms of weight $k$, level $N$, degree $n$ with character $\chi$, and its subspace of cusp forms will be denoted  by $S_k^n(N,\chi)$.
Any $F\in S_k^n(N,\chi)$ has a unique Fourier series expansion 
\begin{equation}\label{eq:fourierseries}
F(Z)=\sum_{T\in\mathcal{J}_n} a_F(T) e(TZ), \ Z \in \mathbb{H}_n.
\end{equation}

Let us recall the following theorem 
\cite[Theorem 2.2]{IK} which will be used latter.  
\begin{theorem}\label{thm:IK}
Assume that $n \geq 2$. Assume that $F \in S_k^{n, {\rm new}}(N,\chi)$ and its Fourier coefficients $a_F(T) = 0$ for all the primitive matrices $T \in \mathcal{J}_n$. Then we have $F = 0$.
\end{theorem} 

\subsection{Jacobi forms and Fourier-Jacobi expansion of Siegel cusp forms}

Let $g, j$ be any two positive integers and $k, N, \chi$ be as above. Let $\mathcal{M}$ be a symmetric, half-integral, positive definite $j \times j$ matrix. Let $\Gamma_0^{g,J}(N):=\Gamma_0^g(N)\ltimes (\mathbb Z^{j,g} \times \mathbb Z^{j,g})$ be the Jacobi group of level $N$. The Jacobi group acts on the set of holomorphic functions 
$\phi:\mathbb{H}_g \times \mathbb{C}^{j,g} \rightarrow \mathbb{C}$ parametrized by the pair $(k,\mathcal{M})$ 
as follows.
\begin{equation*}
\phi |_{k, \mathcal{M}} \gamma (\tau, w) = \det(C\tau+D)^{-k} e(-\mathcal{M} w (C\tau + D)^{-1} C w^t) \phi(\mathcal{M} \tau, w (C\tau+D)^{-1}),
\end{equation*}
for all $\gamma=
\begin{pmatrix} A & B \\ C & D \end{pmatrix} \in \Gamma_0^g(N)$, and 
\begin{equation*}
\phi |_{\mathcal{M}} (\lambda, \mu) (\tau, w) = e(\mathcal{M}(\tau[\lambda^t]+2 \lambda w^t + \mu \lambda^t)) 
\phi(\tau, w+\lambda \tau+\mu), \ {\rm for\ all} \ (\lambda, \mu) \in \mathbb Z^{j,g} \times 
\mathbb Z^{j,g}.
\end{equation*}

\begin{definition}
A Jacobi form of weight $k$, index $\mathcal{M}$, level $N$, degree $g$ with character $\chi$ is a 
holomorphic function 
$\phi:\mathbb{H}_g \times \mathbb{C}^{j,g} \rightarrow \mathbb{C}$ such that
$$
\phi |_{k,\mathcal{M}}\gamma (\tau, w) = \chi(\gamma) \phi (\tau, w)\ {\rm and}\
\phi |_{\mathcal{M}} (\lambda, \mu) (\tau, w) = \phi (\tau, w),
$$
for all $[\gamma, (\lambda, \mu)] \in \Gamma^{g,J}_0(N)$, and if $g=1$ then $\phi$ has a Fourier series expansion of the form
$$
\phi(\tau, w) = \sum_{{n \geq 0, r \in \mathbb{Z}^{j}}\atop{4n - \mathcal{M}^{-1}[r^t]} \geq 0} c(n,r) e(n\tau) e(r w)
$$
at all the cusps. 

Moreover, we say that $\phi$ is a cusp form if it vanishes at all the cusps of $\Gamma_0^{g,J}(N)$.
\end{definition}

For more details on Jacobi forms we refer to \cite{EZ, CZ}. The space of Jacobi forms (resp. Jacobi cusp forms) of weight $k$, index $\mathcal{M}$, level $N$, degree $g$ with character $\chi$ is a finite-dimensional $\mathbb{C}$-vector space, denoted by $J^g_{k,\mathcal{M}}(N,\chi)({\rm resp.}\ J^{g, \rm cusp}_{k,\mathcal{M}}(N,\chi))$. If $g=1$, we also  denote these spaces by $J_{k,\mathcal{M}}(N,\chi)({\rm resp.}\ J^{\rm cusp}_{k,\mathcal{M}}(N,\chi))$. 

Let $F \in S_k^{g+j}(N,\chi)$ be a Siegel cusp form having Fourier series expansion 
of the form \eqref{eq:fourierseries}. Let us write
$$Z=\begin{pmatrix} \tau & w^t \\ w & \tau' \end{pmatrix}, \ \tau \in \mathbb{H}_{g}, \tau'\in \mathbb{H}_j, w \in \mathbb{C}^{j, g}.$$
Then one has the following well-known {\it Fourier-Jacobi expansion} of $F$ \cite{EZ, CZ}.
\begin{equation}\label{eq:FJE}
F(Z)=\sum_{\mathcal{M} \in \mathcal{J}_j} \phi_{\mathcal{M}}(\tau, w) e(\mathcal{M} \tau'),
\end{equation}
where each $\phi_{\mathcal{M}} \in J^{g, {\rm cusp}}_{k,\mathcal{M}}(N, \chi)$ is a Jacobi cusp form with the following Fourier series expansion.
\begin{equation}\label{eq:fourierphim}
\phi_{\mathcal{M}}(\tau,w)=\sum_{T_1 \in \mathcal{J}_g, r\in \mathbb Z^{g, j}} 
a_F \left(\begin{pmatrix} T_1 & \frac{1}{2} r \\ \frac{1}{2}r^t & \mathcal{M} \end{pmatrix} \right) e(T_1 \tau + r w),
\end{equation}
with the condition that the matrices $\begin{pmatrix} T_1 & \frac{1}{2} r \\ \frac{1}{2}r^t & \mathcal{M} 
\end{pmatrix}$
are positive definite.
For fix $\mathcal{M} \in \mathcal{J}_g$ we denote 
$a_F\left(\begin{pmatrix} T_1 & \frac{1}{2} r \\ \frac{1}{2}r^t & \mathcal{M} \end{pmatrix}\right)$ by $c(T_1, r)$, then \eqref{eq:fourierphim} takes the following form. 
\begin{equation}\label{eq:fourierphim1}
\phi_{\mathcal{M}}(\tau,w)=\sum_{T_1 \in  \mathcal{J}_g} \sum_{r\in \mathbb Z^{g,j}} c(T_1,r) e(T_1 \tau + r w),
\end{equation}
where the coefficients $c(T_1,r)$ is non-zero only if the matrix $\begin{pmatrix} T_1 & \frac{1}{2} r \\ \frac{1}{2}r^t & \mathcal{M} \end{pmatrix}$ is positive definite.

\subsection{The theta decomposition of Jacobi cusp forms}

The invariance of the Jacobi cusp form $\phi \in J_{k,\mathcal{M}}^{g, {\rm cusp}}(N, \chi)$
with respect to the subgroup $\{I_g\} \ltimes \left( \mathbb{Z}^{j,g} \times \mathbb{Z}^{j,g}\right)$ of 
$\Gamma_0^{g,J}(N)$ is equivalent to the set of following relations among Fourier coefficients.
\begin{equation}\label{f rel among fourier coeffs}
c(T, r) = c\left(T+2^{-1}(r\lambda+\lambda^t r^t)+ \mathcal{M}[\lambda], r+2 \lambda^t \mathcal{M} \right),
\end{equation}
for any $\lambda \in \mathbb{Z}^{j, g}$. 
For any $r \in \mathbb{Z}^{g,j}$ and $\mathcal{T} \in \mathcal{J}_g$, let
$c_r(\mathcal{T}) := c(T,r)$ if $\mathcal{T} = 4 (\det \mathcal{M}) T - \mathcal{M}^*[r^t]$ for some $T \in \mathcal{J}_g$, otherwise $c_r(\mathcal{T}) := 0$. 

Let us denote the quotient group $\mathbb{Z}^{g,j} / \mathbb{Z}^{g,j} 2 \mathcal{M}$ by  $\mathcal{I}_{g,j}$.
If $g=1$, we simply write $\mathcal{I}_j$ in place of $\mathcal{I}_{1,j}$.
For any $r \in \mathcal{I}_{g,j}$, consider the Jacobi theta series
\begin{equation*}\label{f theta ser}
\Theta_{\mathcal{M},r}(\tau, w) = \sum_{\lambda \in \mathbb{Z}^{j, g}} e\left(
\mathcal{M} [\lambda + (2\mathcal{M})^{-1} r^t] \tau + 2 \mathcal{M} (\lambda + (2\mathcal{M})^{-1} r^t) w^t \right), \ (\tau,w) \in 
\mathbb{H}_g\times \mathbb{C}^{j,g}.
\end{equation*}
Consider the column vector 
$\Theta(\tau,w)=(\dots, \Theta_{\mathcal{M}, r}(\tau, w), \dots)_{r \in \mathcal{I}_{g,j}}^t$. 
From \cite{GS, TS}, we have a unitary map $\rho$ on $Sp_g(\mathbb{Z})$ such that
\begin{equation*}\label{eq:thetatrans}
\Theta(\gamma (\tau,w)):=\Theta\left((A\tau+B) (C\tau+D)^{-1}, w (C\tau+D)^{-1} \right)=
\det(C\tau + D)^{1/2} \rho(\gamma) \Theta (\tau,w), 
\end{equation*}
for any $\gamma = \begin{pmatrix} A & B \\ C & D \end{pmatrix}
\in Sp_g(\mathbb{Z})$.
A formal manipulation of the Fourier series of $\phi$ (given by \eqref{eq:fourierphim1}) and using 
\eqref{f rel among fourier coeffs} yields the following theta decomposition of the Jacobi cusp form $\phi$.
\begin{equation}\label{f theta decomp}
\phi(\tau,w) = \sum_{r \in {\mathcal{I}_{g,j}}} h_r (\tau)
\Theta_{\mathcal{M}, r} (\tau, w),
\end{equation}
where 
$$
h_r(\tau) = \sum_{T \in \mathcal{J}_g} c_r(T) e\left(\frac{1}{4 \det \mathcal{M}} T \tau\right).
$$

Now we record a proposition giving transformation properties of the theta components $h_r$ in case $g=1$, which will be used later. Though the proof can be found in the literature but for completion we also give a sketch 
of the proof here.  
\begin{prop}\label{prop:hmutrans}
Let $j \geq 1$. Let $\mathcal{M}$ be a symmetric, half-integral, positive definite $j \times j$ matrix and 
$\phi \in J^{\rm cusp}_{k,\mathcal{M}}(N,\chi)$ with the theta decomposition \eqref{f theta decomp}.
Then for any $r \in \mathcal{I}_j$, we have $h_r \in S_{k-\frac{j}{2}}(\Gamma(N'))$, where $N'=4Nd'$ and 
$d'=\frac{1}{2}\det(2\mathcal{M})$ or $\det(2\mathcal{M})$ depending on $j$ is odd or even respectively. 
\end{prop}
\begin{proof}
The proposition follows by using the following transformation properties for the theta series 
\cite[\S 2.3, Proposition 2.3.1]{AA} together with the fact that $\phi \in J^{\rm cusp}_{k,M}(N,\chi)$. 
For $j$ odd, we have
\begin{equation*}\label{eq:thetatransodd}
\begin{split}
\theta_{\mathcal{M},r}\left(\frac{a\tau+b}{c\tau+d}, \frac{z}{c\tau+d}\right)=\left(\frac{c}{d}\right)\left(\frac{-4}{d}\right)^{-\frac{1}{2}}\left(\frac{(-1)^{j/2} d'}{d}\right) & e\left(bd\mathcal{M}[r^t]+\mathcal{M}[z] \frac{c}{c\tau+d}\right)\\
&(c\tau+d)^{j/2} \theta_{\mathcal{M},dr}(\tau,z), \ \ \begin{pmatrix} a&b\\c&d \end{pmatrix} \in \Gamma_0(4d').
\end{split}
\end{equation*} 
For $j$ even, we have
\begin{equation*}\label{eq:thetatranseven}
\theta_{\mathcal{M},r}\left(\frac{a\tau+b}{c\tau+d}, \frac{z}{c\tau+d}\right)=\left(\frac{(-1)^{j/2} d'}{d}\right)  e\left(bd\mathcal{M}[r^t]+\mathcal{M}[z] \frac{c}{c\tau+d}\right)
(c\tau+d)^{j/2} \theta_{\mathcal{M},d r}(\tau,z), 
\end{equation*}
for any matrix $ \begin{pmatrix} a&b\\c&d \end{pmatrix} \in \Gamma_0(d')$. 
\end{proof}

\section{Statements of the intermediate results}
To prove Theorems \ref{thm:1} and \ref{thm:3} we need to work on theta decomposition of Jacobi cusp forms and Fourier-Jacobi decomposition of Siegel cusp forms of arbitrary degree. More precisely, we need to prove the following three propositions. The first one gives non-vanishing of the theta components and generalizes 
\cite[Proposition 1.1]{M} to any degree whereas the next two are about  
non-vanishing of certain Fourier-Jacobi coefficients. These results may have some independent interest as well.  
We state our results below and prove them in \S 5 after proving some lemmas from the theory of quadratic forms in the next section. 

\begin{prop}\label{prop:main}
Let $g, k, N$ be positive integers. Assume that $N$ is odd. Let $\chi$ be any Dirichlet character modulo $N$ such that $\chi(-1)=(-1)^k$. Let $p$ be an odd prime such that $\gcd(p, N)=1$. Let 
$\phi \in J^{g, {\rm cusp}}_{k, p} (N, \chi)$ be a non-zero Jacobi cusp form. Then all the theta components 
$\phi_\mu(\tau), \mu \in \mathcal{I}_{g,1} (=\mathbb{Z}^{g,1}/2p\mathbb{Z}^{g,1})$, in the theta decomposition
\begin{equation}\label{eq:propmain}
\phi(\tau, z) = \sum_{\mu \in \mathcal{I}_{g,1}} \phi_{\mu} (\tau) \theta_{p, \mu} (\tau,z) 
\end{equation}
of $\phi$ $($given by \eqref{f theta decomp}$)$ are non-zero. 
\end{prop}

\begin{prop}\label{prop:1}
Let $g, k, N$ be positive integers and $\chi$ be a Dirichlet character modulo $N$. Let 
$F \in S_k^{g+1, {\rm new}}(N, \chi)$ be a non-zero Siegel cusp form with Fourier-Jacobi expansion 
$($governed by \eqref{eq:FJE}$)$ 
\begin{equation}\label{eq:prop1}
F(Z) = \sum_{m \geq 1} \phi_m(\tau, z) e(m\tau'), \ Z = \begin{pmatrix} \tau&z^t\\z&\tau' \end{pmatrix},\ \tau \in \mathbb{H}_g, z \in \mathbb{C}^{g}, \tau' \in \mathbb{H}.
\end{equation}
Then there exist infinitely many odd primes $p$ such that $\phi_p \not \equiv 0$.
\end{prop}

\begin{prop}\label{prop:2}
Let $g, k, N, \chi$ be as in Proposition \ref{prop:main}. Assume that $g \geq 2$. 
Let $F \in S_k^{g+1, {\rm new}}(N, \chi)$ be a non-zero Siegel cusp form with Fourier-Jacobi expansion 
$($governed by \eqref{eq:FJE}$)$ 
\begin{equation}\label{eq:prop2}
F(Z)=\sum_{\mathcal{M} \in \mathcal{J}_g} \psi_{\mathcal{M}} (\tau, z) e(\mathcal{M}\tau'), \ Z = \begin{pmatrix} \tau&z^t\\z&\tau' \end{pmatrix}, \ \tau \in \mathbb{H}, z \in \mathbb{C}^{g,1}, \tau' \in \mathbb{H}_g.
\end{equation}
Then there exist infinitely many primitive, symmetric, half-integral, positive definite 
$g \times g$ matrices $\mathcal{M}$ such that $\psi_{\mathcal{M}} \neq 0$.
\end{prop}

\section{Some key lemmas from quadratic forms}
Let $Q(x)=\sum_{1\leq i \leq j\leq n}c_{ij}x_i x_j$ be a quadratic form in $n (\geq 2)$ variables with integral coefficients. Write $Q(X)=\frac{1}{2} x^t A x$, where $x$ is the column vector $(x_j)$ and $A$ is the 
$n \times n$ symmetric matrix $(a_{ij})$ with entries $a_{ii}=2c_{ii}$ and $a_{ij}=c_{ij}$ for $i<j$. 
We say that $Q$ is {\it primitive} if the gcd of all the coefficients $c_{ij}$ is $1$, in other words the half-integral matrix $\frac{1}{2} A$ is primitive. We have the following lemma.
\begin{lemma}\label{lemma:1}
Let $Q$ be a primitive quadratic form as above and $N$ be a positive integer. Then $Q$ represents an integer $m\ ($over $\mathbb{Z})$ which is relatively prime to $N$.
\end{lemma}
\begin{proof}
Let $Q(x)=\frac{1}{2}x^{t}Ax$ be a primitive quadratic form in $n$ variables.
Without loss of generality we assume that $N$ is squarefree, that is,
$N = p_{1}p_{2} \dots p_{k}$ for distinct primes $p_{i}$.
Since $Q$ is primitive, $Q_i(x):=Q(x) \pmod{p_i}$ is a non-zero quadratic form for each $i=1,2,\dots k$.
For any such $i$, let $\tilde{x}_i=(\tilde{x}_{i1},\tilde{x}_{i2},\dots,\tilde{x}_{in}) \in \mathbb{Z}^{n}$ be such that    
$Q_{i}(\tilde{x}_{i}) \not\equiv 0 \pmod{p_{i}}$. 
By using Chinese remainder theorem we get $y_1,y_2, \dots y_n$
unique modulo $N$ such that 
$$y_j \equiv \tilde{x}_{ij} \pmod {p_i}, i=1,2, \dots n, j=1,2,\dots k.$$
Then $y=(y_1,y_2,\dots y_n) \in \mathbb{Z}^n$ such that $Q(y)\not\equiv 0 \pmod {p_i}$ for any 
$i=1,2 \dots k$ and hence $Q(y)$ is co-prime to $N$.
\end{proof}

\begin{lemma}\label{lemma:2}
Let $Q$ be a primitive, positive definite quadratic form with the same notations as above. Then $Q$ represents infinitely many odd primes over $\mathbb{Z}$.
\end{lemma}
\begin{proof}
In the case $n=2$, we claim to the classical result of Weber \cite{HW} to conclude the lemma. 
Suppose $n \geq 3$.
In view of \cite[Theorem 1]{WD} and \cite[Corollary 11.3]{I} it is sufficient to show that there exists infinitely many odd primes $p$ such that for each $p$ the congruence $Q(x)\equiv p \pmod{2^7(\det A)^3}$ is solvable over    
$\mathbb{Z}$. From Lemma \ref{lemma:1} we get that $Q$ represents an integer $m$ co-prime to $2^7(\det A)^3$. From Dirichlet theorem for primes in arithmetic progression we get infinitely many odd primes $p$ in the arithmetic progression $m+l(2^7(\det A)^3), l=0,1,2,\dots$. Since all these $p$ are congruent to $m$ modulo $2^7(\det A)^3$, we get the required congruence.      
\end{proof}
		
\begin{lemma}\label{lemma:3}
Let $A= \begin{pmatrix} T_{1} & r/2 \\ r^t/2 & m \end{pmatrix}$ be a symmetric, half-integral, positive definite, primitive $n \times n \ (n \geq 2)$ matrix. Then there exist infinitely many odd primes $p$ such that for each such prime $p$ there exists a matrix $A_p \in SL_n(\mathbb Z)$ with the property that
$A'=A_{p}^t A A_p=\begin{pmatrix} T'_{1} & r'/2 \\ {r'}^{t}/2 & p \end{pmatrix}$.
\end{lemma}
\begin{proof}
From Lemma \ref{lemma:2} we know that the quadratic form $Q(x)=x^tAx$ represents infinitely many odd primes $p$. For each such prime $p$ let $x_p=(x_{p1},x_{p2}, \dots x_{pn})$ be such that $Q(x_p)=p$. 
This implies that $\gcd(x_{p1},x_{p2}, \dots x_{pn})=1$ and therefore it can be extended to an 
$SL_n(\mathbb{Z})$ matrix $A_p$ having the last column $x_p$. The matrix $A_p^tAA_p$ has the required property.
\end{proof}

\section{Proof of the intermediate results}

\begin{proof}[Proof of Proposition \ref{prop:main}]
We have the following transformation property of theta series \cite[Page 62]{AK}. For any 
$\alpha \in \mathcal{I}_{g,1}$ and $M=\big(\begin{smallmatrix}A & B\\C & D\end{smallmatrix}\big) \in Sp_{g}(\mathbb{Z})$ there are complex numbers  
$\rho_{\alpha,\beta}(M), \beta \in \mathcal{I}_{g,1}$, such that 
\begin{equation}\label{e1}
\begin{split}
\theta_{p,\alpha}&\left((A\tau + B) (C\tau + D)^{-1},w (C\tau + D)^{-1}\right)\\ 
&=\det(C\tau + D)^{1/2} e\left(p w^{t}w(C\tau + D)^{-1}C \right) \!\! \sum_{\beta \in \mathcal{I}_{g,1} } \!\!\rho_{\alpha,\beta}(M) \theta_{p,\beta}(\tau,w), \
(\tau,w) \in \mathbb{H}_g \times \mathbb{C}^g.
\end{split}
\end{equation}
We have the following transformation property of $\phi$ with respect to the triplet 
$[M,0,0], M \in \Gamma_{0}^{g}(N)$.
\begin{equation*}
\det(C\tau+D)^{-k}  e\left(-p w^{t}w(C\tau + D)^{-1}C \right) \phi(M (\tau,w)) = \chi(M) \phi(\tau, w).
\end{equation*}
Now by writing the theta decomposition \eqref{eq:propmain} of both sides and then by using \eqref{e1}, for any 
$\nu \in \mathcal{I}_{g,1}$ we have
\begin{equation*}
\det(C\tau+D)^{-k} \det(C\tau + D)^{1/2} \sum_{\mu \in \mathcal{I}_{g,1} } \rho_{\mu,\nu}(M) \phi_{\mu}(M \tau)=\chi(M) \phi_{\nu}(\tau).
\end{equation*}
Suppose $\phi_{\nu} \equiv 0$ for some $\nu \in \mathcal{I}_{g,1}$. 
Then we have 
$$
\sum_{\mu \in \mathcal{I}_{g,1} } \rho_{\mu,\nu}(M) \phi_{\mu}(\tau)=
\sum_{T \in \mathcal{J}_g} \left(\sum_{\mu \in \mathcal{I}_{g,1} \atop \mu \mu^{t}+T \in {4p} \mathcal{J}_g}
\rho_{\mu,\nu}(M)c_{\mu}(T)\right) e\left(\frac{T\tau}{4p}\right)=0.
$$ 
Therefore we have
\begin{equation*}\label{e2}
\sum_{\mu \in \mathcal{I}_{g,1} \atop \mu \mu^{t}+T \in {4p} \mathcal{J}_g} \rho_{\mu,\nu}(M) c_{\mu}(T)=0, \ \ \ \forall\ T \in \mathcal{J}_g,\ M \in \Gamma_{0}^{g}(N).
\end{equation*} 
Let $T \in \mathcal{J}_g$. If for any two $\mu, r \in \mathcal{I}_{g,1}$, $T+ \mu \mu^{t}, T+r r^t \in {4p}\mathcal{J}_g$ then $r \equiv \pm \mu \pmod{2p}$. Note that if all the entries of $\mu$ are $p$ or $2p$ then 
$\mu = -\mu$ in $\mathcal{I}_{g,1}$. Since $c_{-\mu}(T)=(-1)^{g k} \chi((-1)^g) c_\mu(T)$, under the assumption on $\chi$, we get
$$(\rho_{\mu,\nu}(M) + \rho_{-\mu,\nu}(M)) c_{\mu}(T)=0,$$
for any $T \in \mathcal{J}_g, \mu \in \mathcal{I}_{g,1}$ with $T+\mu \mu^t \in {4p}\mathcal{J}_g$, and 
$M \in \Gamma_0^g(N)$.
Thus for any such pair $T$ and $\mu$, to show $c_{\mu}(T)=0$ it is sufficient to 
find a matrix $M' \in \Gamma_{0}^{g}(N)$ such that 
\begin{equation}\label{l1}
\rho_{\mu,\nu}(M')\neq 0	
\end{equation} 
if all the entries of $\mu$ are $p$ or $2p$, and otherwise a matrix 
$M'' \in \Gamma_{0}^{g}(N)$ such that 
\begin{equation}\label{l2}
\rho_{\mu,\nu}(M'') + \rho_{-\mu,\nu}(M'') \neq 0.
\end{equation} 

Next we establish the existence of such matrices $M', M''$. 
Set $j(M; \tau) = \det (C \tau +D)$ for $M = \big(\begin{smallmatrix}A & B\\C & D\end{smallmatrix}\big) \in Sp_{g}(\mathbb{Z})$. For any two matrices $M_{1},M_{2}\in Sp_{g}(\mathbb{Z})$, by comparing the two equations obtained by using \eqref{e1} once for the product $M_1M_2$, secondly for $M_1$ and then for $M_2$, we have
\begin{equation}\label{e3}
\rho_{\alpha,\beta}(M_{1}M_{2})=\frac{j(M_{1};M_{2}\tau)^{1/2} j(M_{2};\tau)^{1/2}}{j(M_1M_2;\tau)^{1/2}} \sum_{\nu \in \mathcal{I}_{g,1} }\rho_{\alpha, \nu}(M_{1})\rho_{\nu, \beta}(M_{2}).
\end{equation}
Notice that the first term, that is, the ratio of cocycles $j(M; \tau)$ in the right hand side of the above equation is independent of $\tau$, and therefore its value can be obtained by evaluating it at any $\tau \in \mathbb{H}_{g}$.
Let us denote this ratio by $J(M_1, M_2; \tau)$.
 
Let $M_{1}=\left(\begin{array}{cc} 0 & -E\\ E & 0 \end{array}\right), M_{2}=\left(\begin{array}{cc} E & E\\ 0 & E \end{array}\right)$, where $E$ is the identity matrix of order $g$. By using the transformation properties of the theta functions \cite [Lemma 3.2]{CZ}, we have
\begin{equation}\label{e4}
\rho_{\alpha,\beta}(M_{1})= (2p)^{-g/2} \kappa(M_{1}) e\left(\frac{-\beta^{t} \alpha}{2p}\right), \ \rho_{\alpha,\beta}(M_{2}^{m})= \left\{\begin{array}{cc}
e(\frac{m\alpha^{t}\alpha}{4p}) & {\rm if} \ \alpha=\beta \\ 
0 & {\rm otherwise} \end{array} \right.,
\end{equation}
where $\kappa(M_1)$ is some $8$-th roots of unity and $m$ is any integer. By taking $\tau = i E$, we get that the ratio of cocycles $J(M_1, M_2; \tau)=1$. By using \eqref{e4} together with \eqref{e3}, for any integer $m \in \mathbb{Z}$ we have
\begin{equation*}
\rho_{\alpha,\beta}(M_{2}^{m}M_{1})=\rho_{\alpha, \alpha}(M_{2}^{m}) \rho_{\alpha,\beta}(M_{1})= (2p)^{-g/2} \kappa(M_{1})  e\left(\frac{m\alpha^{t}\alpha- 2\beta^{t}\alpha}{4p}\right).
\end{equation*}
Again by taking $\tau=i E$ one has $J(M_2^l M_1, M_2^m M_1; \tau) =1$ for any $l, m \in \mathbb{Z}$ with $m \geq 1$. Therefore by using \eqref{e3} we have 
\begin{align*}
\rho_{\alpha,\beta}(M_{2}^{l}M_{1}M_{2}^{m}M_{1}) &= \sum_{\nu \in \mathcal{I}_{g,1} } \rho_{\alpha,\nu}(M_{2}^{l}M_{1}) \rho_{\nu, \beta}(M_{2}^{m}M_{1}) \\ &= (2p)^{-g} \kappa^{2}(M_{1}) e\left(\frac{l\alpha^{t} \alpha}{4p}\right) \sum_{\nu \in \mathcal{I}_{g,1} } e\left(\frac{m \nu^{t}\nu-2\nu^{t}\alpha-2\beta^{t}\nu}{4p}\right)
\end{align*}
for all $l, m \in \mathbb{Z}$ with $m \geq 1$. Next, by taking $\tau = (i/m) E$ one has $J(M_1, M_2^l M_1M_2^m M_1; \tau) =1$ for any $l, m \in \mathbb{Z}$ with $l, m \geq 1$.
Therefore again by using \eqref{e3}, for such $l,m$ we have
\begin{dmath}\label{e5}
\rho_{\alpha,\beta}(M_{1}M_{2}^{l}M_{1}M_{2}^{m}M_{1}) = (2p)^{-3g/2} \kappa^{3}(M_{1}) \sum_{\lambda \in \mathcal{I}_{g,1} } e\left(\frac{l\lambda^{t}\lambda-2\lambda^{t}\alpha}{4p}\right) \times \\  \sum_{\nu \in \mathcal{I}_{g,1} }e\left(\frac{m\nu^{t} \nu-2\nu^{t} \lambda- 2\beta^{t}\nu}{4p}\right).
\end{dmath}
Now let us evaluate the exponential sum appearing in the above equation. 
Let us write 
$$\alpha = \left(\begin{array}{cc} \alpha_{1}  \\ \alpha_{2}  \\ \vdots \\ \alpha_{g} \end{array}\right),
\beta= \left(\begin{array}{cc} \beta_{1}  \\ \beta_{2}  \\ \vdots \\ \beta_{g} \end{array}\right), \lambda = \left(\begin{array}{cc} \lambda_{1}  \\ \lambda_{2}  \\ \vdots \\ \lambda_{g} \end{array}\right), 
\nu= \left(\begin{array}{cc} \nu_{1}  \\ \nu_{2}  \\ \vdots \\ \nu_{g} \end{array}\right).$$
Then we have
\begin{dmath}\label{e6}
\sum_{\nu \in \mathcal{I}_{g,1} }e\left(\frac{m\nu^{t} \nu-2\nu^{t} \lambda- 2\beta^{t}\nu}{4p}\right)\\ =
\sum_{\nu_{1} \pmod {2p}} e\left(\frac{m\nu_{1}^{2}- 2\nu_{1}\lambda_{1}-2\beta_{1}\nu_{1}}{4p}\right)  \ldots \sum_{\nu_{g} \pmod {2p}} e\left(\frac{m\nu_{g}^{2}- 2\nu_{g}\lambda_{g}-2\beta_{g}\nu_{g}}{4p}\right).
\end{dmath}
For any $\nu_j, j=1,2,..., g$, by using \cite[eqn (22), p.13]{M},  we get
\begin{dmath}\label{e7}
\sum_{\nu_{j} \pmod{2p}} e\left(\frac{m\nu_{j}^{2}-2\nu_{j}(\lambda_{j}+\beta_{j})}{4p}\right) \\ = \epsilon_{p}\sqrt{p} \left(\frac{m}{p}\right) i^{(\lambda_{j}+\beta_{j})^{2}p} \left(1+(-1)^{\lambda_{j}+\beta_{j}} e\left(\frac{mp}{4}\right)\right) e\left(\frac{-(\lambda_{j}+\beta_{j})^{2}}{4p}\right),
\end{dmath}
for any $m \geq 1, p \nmid m$. Here $\epsilon_p$ is $1$ or $i$, according to $p\equiv 1$ or $-1 \pmod{4}$. 
Now by combining \eqref{e7} with \eqref{e6} and then by using it in \eqref{e5} we get 
\begin{dmath}\label{e9}
\rho_{\alpha,\beta}(M_{1}M_{2}^{l}M_{1}M_{2}^{m}M_{1})\\ =  (2p)^{-3g/2} \kappa^{3}(M_{1}) \sum_{\lambda \in \mathcal{I}_{g,1} } e\left(\frac{l\lambda^{t}\lambda-2\lambda^{t}\alpha}{4p}\right) \times \\ 
\hspace*{2cm} \prod_{j=1}^{g} \left(\epsilon_{p}\sqrt{p} \left(\frac{m}{p}\right) i^{(\lambda_{j}+\beta_{j})^{2}p} \left(1+(-1)^{\lambda_{j}+\beta_{j}} e\left(\frac{mp}{4}\right)\right) e\left(\frac{-(\lambda_{j}+\beta_{j})^{2}}{4p}\right)\right)\\
=(2p)^{-3g/2} p^{g/2} \epsilon_{p}^{g} \left(\frac{m}{p}\right)^{g} \kappa^{3}(M_{1})\times \\ \hspace*{2cm} \prod_{j=1}^g
\sum_{\lambda_{j} \pmod{2p}} e\left(\frac{l\lambda_{j}^{2}-2\lambda_{j}\alpha_{j}}{4p}\right) \left( i^{(\lambda_{j}+\beta_{j})^{2}p} \left(1+(-1)^{\lambda_{j}+\beta_{j}} e\left(\frac{m p}{4}\right)\right) e\left(\frac{-(\lambda_{j}+\beta_{j})^{2}}{4p}\right)\right).
\end{dmath}
For each $j \in \{1,2,..., n\}$, by using \cite[Eq. 24, 25 and 26]{M} we get 
\begin{dmath}\label{e11}
\sum_{\lambda_{j} \pmod{2p}} i^{(\lambda_{j}+\beta_{j})^{2}p} \left(1+(-1)^{\lambda_{j}+\beta_{j}} e\left(\frac{mp}{4}\right)\right)  e\left(\frac{l\lambda_{j}^{2}-2\lambda_{j}\alpha_{j}}{4p}-\frac{(\lambda_{j}+\beta_{j})^{2}}{4p}\right)= \epsilon_{p} \sqrt{p} \left(\frac{l-1}{p}\right) i^{(\alpha_{j}+\beta_{j})^{2}p} i^{\beta_{j}^{2}p} e\left(\frac{-(\alpha_{j}+\beta_{j})^{2}}{4p}\right) \left(1+(-1)^{\beta_{j}} e\left(\frac{mp}{4}\right)+ (-1)^{\alpha_{j}} e\left(\frac{lp}{4}\right) \left(1- (-1)^{\beta_{j}}e\left(\frac{mp}{4}\right)\right)\right), \ \ p \nmid (l-1).
\end{dmath}
By using \eqref{e11} with \eqref{e9}, we have
\begin{dmath*}\label{e12}
\rho_{\alpha,\beta}(M_{1}M_{2}^{l}M_{1}M_{2}^{m}M_{1}) = 2^{-3g/2} p^{-g/2} \epsilon_{p}^{2g} \left(\frac{m}{p}\right)^{g}  \left(\frac{l-1}{p}\right)^g  \kappa^{3}(M_{1}) \prod_{j=1}^{g} i^{(\alpha_{j}+\beta_{j})^{2}p} 
i^{\beta_{j}^{2}p} e\left(\frac{-(\alpha_{j}+\beta_{j})^{2}}{4p}\right) \\ \hspace*{5cm} \times \left(1+(-1)^{\beta_{j}} e\left(\frac{mp}{4}\right)+ (-1)^{\alpha_{j}} e\left(\frac{lp}{4}\right) \left(1- (-1)^{\beta_{j}}e\left(\frac{mp}{4}\right)\right)\right), \ \ p \nmid (l-1) m.
\end{dmath*}
Since $(-1)^{-\alpha_j}=(-1)^{\alpha_j}$ and $(-\alpha_j+\beta_j)^2 \equiv (\alpha_j+\beta_j)^2 \pmod{4}$, \eqref{l1} and \eqref{l2} will hold if we could choose positive integers 
$l, m$ such that $p \nmid (l-1) m, M_{1}M_{2}^{l}M_{1}M_{2}^{m}M_{1}
\in \Gamma_0^g(N)$, and
\begin{equation}\label{e12}
\left(1+(-1)^{\beta_{j}} e\left(\frac{mp}{4}\right)+ (-1)^{\alpha_{j}} e\left(\frac{lp}{4}\right) \left(1- (-1)^{\beta_{j}}
e\left(\frac{mp}{4}\right)\right)\right) \neq 0 
\end{equation}
for all $\alpha_{j}, \beta_{j} \in \mathbb{Z}/2p\mathbb{Z}$. 
As $N$ is odd and $\gcd(p,2N)=1$, we can choose integers $l\equiv m \equiv 0 \pmod{4}$ with 
$l m \equiv 1 \pmod{N}$ and $p \nmid (l-1) m$. Then 
$M_{1}M_{2}^{l}M_{1}M_{2}^{m}M_{1} \in \Gamma_{0}^{g}(N)$ and the left hand side of \eqref{e12} 
is equal to 
$$1+ (-1)^{\beta_{j}}+ (-1)^{\alpha_{j}}- (-1)^{\alpha_{j}+\beta_{j}}= \begin{cases}
2,    \hspace{0.5cm}   {\rm if} \  \beta_{j}\equiv 0\pmod{2}\\ 
(-1)^{\alpha_{j}}2, \hspace{0.2cm}   {\rm if} \ \beta_{j}\equiv 1\pmod{2}. \end{cases},$$
which are never zero for any $j \in \{1,2,\ldots, g\}$.
\end{proof}

\begin{proof}[Proof of Proposition \ref{prop:1}]
Since $F$ is non-zero, by using Theorem \ref{thm:IK} we get a primitive $(g+1) \times (g+1)$ matrix 
$T=\begin{pmatrix}T_{1} & r/2\\ r^t/2 & m \end{pmatrix}, T_1 \in \mathcal{J}_g, m \geq 1$ such that 
$a_F(T) \neq 0$ and hence 
$\phi_m\neq 0$. By using Lemma \ref{lemma:3} we get infinitely many odd primes $p$ such that 
$a_F(A_p^tTA_p)$ are the Fourier coefficients of $\phi_p$. We have $a_F(A_p^tTA_p) = a_F(T) \neq 0$ and hence $\phi_p \neq 0$.
\end{proof}

\begin{proof}[Proof of Proposition \ref{prop:2}]
Let us first consider the Fourier-Jacobi expansion of $F$ as in \eqref{eq:prop1}. Then by using Proposition \ref{prop:1} we get infinitely many odd primes $p$ such that $\gcd(p, N)=1$ and $\phi_p \neq 0$. For any such $\phi_p$, by using Proposition \ref{prop:main} we get that all the theta coefficients 
$\phi_{p,\mu} (\mu \in \mathcal{I}_{g,1})$ of $\phi_p$ are non-zero. In particular, $\phi_{p, \mu_1} \neq 0$, where $\mu_1$ is the column vector whose all entries are $1$. We have
$$
\phi_{p, \mu_1}(\tau)=\sum_{{T \in \mathcal{J}_g}\atop T + \mu_1\mu_1^t \in {4p}\mathcal{J}_g} c_{\mu_1}(T) 
e\left(\frac{T\tau}{4p}\right).
$$
Since $\phi_{p,\mu_1} \neq 0$, there exists a $T \in \mathcal{J}_g$ such that $c_{\mu_1}(T) \neq 0$. This means we have a $\mathcal{T} \in \mathcal{J}_g$ with $T = 4p\mathcal{T}-\mu_1 \mu_1^t$ such that 
$$c_{\mu_1}(T) = a_F\left(\begin{pmatrix} \mathcal{T} & \mu_1/2 \\ \mu_1^t/2 & p\end{pmatrix}\right) \neq 0.$$
Write $\mathcal{T}=(t_{ij})$ and define $\mathcal{M}_1$ to be the $g \times g$ matrix 
$$\mathcal{M}_1=\begin{pmatrix} t_{22} & \cdots & t_{2g} & 1/2 \\ \vdots & \vdots & \vdots & \vdots\\
t_{g2} & \cdots & t_{gg} & 1/2 \\ 1/2 & \cdots & 1/2 & p \end{pmatrix}.$$
Clearly $\mathcal{M}_1 \in \mathcal{J}_g$ which is primitive. In the Fourier-Jacobi decomposition \eqref{eq:prop2}, the coefficient $\psi_{\mathcal{M}_1}$ will be non-zero as $a_F\left(\begin{pmatrix} \mathcal{T} & \mu_1/2 \\ \mu_1^t/2 & p\end{pmatrix}\right) \neq 0$. 
\end{proof}

\section{Proofs of Theorem \ref{thm:1} and Theorem \ref{thm:3}}

\subsection{Proof of Theorem \ref{thm:1}}

Let $g, k, N, \chi$ be as in the theorem. Let $F \in S_k^{g+1}(N, \chi)$ with the Fourier series expansion 
$$F(Z)= \sum_{T \in \mathcal{J}_{g}} a_F(T) e(T Z).$$
By using Proposition \ref{prop:2} we get infinitely many primitive matrices $\mathcal{M} \in \mathcal{J}_g$ 
such that 
$\phi_{\mathcal{M}} \not\equiv 0$, where $\phi_{\mathcal{M}} \in J^{\rm cusp}_{k,\mathcal{M}}(N,\chi)$ 
are Jacobi cusp forms 
appeared in the Fourier-Jacobi decomposition \eqref{eq:prop2} of $F$. 
Fix any such $\mathcal{M}$. Let 
$$\phi_{\mathcal{M}} = \sum_{\mu\in \mathcal{I}_g} h_\mu 
\theta_{\mathcal{M},\mu}$$ 
be the theta decomposition of $\phi_{\mathcal{M}}$. 
Then there exists a $\mu \in \mathcal{I}_{g}$ such that $h_\mu \not \equiv 0$.
By using Proposition \ref{prop:hmutrans} we have $h_\mu \in S_{k-\frac{g}{2}}(\Gamma(N'))$. By using 
\cite[Proposition 7.3.3, Theorem 7.3.4]{CS} for integral weight and similar arguments in half-integral weight 
case, we have
$$h_\mu(N' \tau) \in S_{k-g/2}(\Gamma_1({N'}^2)) = \oplus \ S_{k-g/2}(\Gamma_0({N'}^2),\psi),$$ 
where the direct sum is over all Dirichlet characters $\psi$ modulo ${N'}^2$. We have 
$$h_\mu(N' \tau)=\sum_{{n \geq 1}\atop{n \equiv -\mathcal{M}^*[\mu^t] \pmod{4 \det \mathcal{M}}}} c_\mu(n) e(n\tau),
\ \ c_\mu(n)=a_F\left(\begin{pmatrix} \frac{n+\mathcal{M}^*[\mu^t]}{4 \det \mathcal{M}} & \mu/2 \\ \mu^t/2 & \mathcal{M} \end{pmatrix}\right).$$ 
Let $\hat{c}_\mu(n) = \frac{c_\mu(n)}{n^{(\lambda-1)/2}}, \lambda = k-\frac{g}{2}$.
By using the Fourier coefficients bound due to Deligne \cite{PD} and Bykovskii \cite{VB}, we have the following
\cite[Theorem 3.4]{HKLW}. 
\begin{equation}\label{boundcoefficients}
\hat{c}_\mu(n) \ll_\epsilon
\begin{cases}
n^{\epsilon} &  \text{if}\ g \ \text{is even},\\
n^{\frac{3}{16}+\epsilon}  & \text{if} \ g\ \text{is odd}.\\
\end{cases}
\end{equation}
By using the bound \cite[Corollary 3.5]{HKLW} for average sum of the coefficients, we have  
\begin{equation}\label{estimates}
\sum_{n \leq x} \hat{c}_\mu(n) \ll_\epsilon
\begin{cases}
n^{\frac{1}{3}+\epsilon} &  \text{if}\ g \ \text{is even},\\
n^{\frac{19}{48}+\epsilon}  & \text{if} \ g\ \text{is odd}.\\
\end{cases}
\end{equation}
By applying the Rankin-Selberg method to $h_\mu(N'\tau)$ (see \cite[Page 357, Theorem 1]{RAR}, \cite[Eq. 1.14]{AS}), we get the following estimate.
\begin{equation}\label{eq:estimates2}
\sum_{n \leq x} \hat{c}_\mu^2(n) = K_\mu x + O(x^{3/5+\epsilon}),
\end{equation}
for any $\epsilon > 0$. Here $K_\mu$ is a constant. 

Now we use \cite[Theorem 2.1]{HKLW} which gives a criteria for sign changes in a
sequence under certain conditions. By using the estimates given by \eqref{boundcoefficients}, \eqref{estimates}, \eqref{eq:estimates2} with \cite[Theorem 2.1]{HKLW} we get that 
$\hat{c}_\mu(n)$ changes sign at least once for $n \in (x, x+x^{3/5}]$ for sufficiently large $x$. 
This gives us the claimed sign change results for the coefficients $a_F(T)$ in the theorem.
 
\subsection{Proof of Theorem 2}

We follow the proof of Theorem \ref{thm:1} presented in the previous subsection with the assumption that 
$g+1 = 2$, that is, $g=1$. We do not need Proposition \ref{prop:2} to handle this case.

Let $F \in S_k^2(N, \chi)$ be a non-zero Siegel cusp form with real Fourier coefficients as in the statement of 
the theorem. By using Proposition \ref{prop:1}, we get infinitely many odd primes $p$ such that
$\phi_p \in J^{\rm cusp}_{k,p}(N,\chi)$, Jacobi cusp forms in the Fourier-Jacobi decomposition of $F$, 
are not trivially zero. Fix any such prime $p$. By using Proposition \ref{prop:main} for $\phi_p$, we get that 
none of the theta components 
$h_r, r \in \mathcal{I}_{1} (=\mathbb{Z}/2p\mathbb{Z})$ of $\phi_p$ is zero. 
Also, by using Proposition \ref{prop:1} 
we get that each component $h_r$
satisfy the following transformation 
property with respect to the principal congruence subgroup $\Gamma(4Np)$.
\begin{equation*}\label{eq:phirtrans}
h_r(\gamma \tau) = h_r \left(\frac{a\tau+b}{c\tau+d}\right) = 
\left(\frac{c}{d}\right)(c\tau+d)^{k-1/2} h_r(\tau), \ 
\gamma = \begin{pmatrix} a&b\\c&d \end{pmatrix} \in \Gamma (4Np).
\end{equation*}
Moreover, by using Proposition \ref{prop:hmutrans} each $h_r$ is a cusp form of weight $k-\frac{1}{2}$ for 
the group $\Gamma(4Np)$.
Now onwards we follow the proof of Theorem \ref{thm:1} to conclude the claimed result. 

\bigskip

\noindent {\bf Acknowledgement:} 
We would like to thank Prof. Todd Cochrane for some useful conversations about quadratic forms.
The research of the first author was partially supported by the DST-SERB grant ECR/2016/001359.

\end{document}